\documentclass[11pt,a4paper,reqno]{amsart}
\usepackage{amsfonts,caption,amsmath,amssymb,enumerate,latexsym,xcolor,amsthm,multicol,cases,tikz,empheq}
\usepackage{hyperref}
\usetikzlibrary{calc}

\newcommand\AGaL{\mathrm{A\Gamma L}}\newcommand\AGL{\mathrm{AGL}}
\newcommand\bbF{\mathbb{F}}

\newcommand\Cay{\mathrm{Cay}}\newcommand\Cos{\mathrm{Cos}}

\newcommand\FF{\mathbb{F}}

\newcommand\GL{\mathrm{GL}}

\newcommand\Sym{\mathrm{Sym}}
\newcommand\Sy{S}

\newtheorem{theorem}{Theorem}[section]
\newtheorem{corollary}[theorem]{Corollary}
\newtheorem{lemma}[theorem]{Lemma}
\newtheorem{proposition}[theorem]{Proposition}
\theoremstyle{definition}

\newtheorem{question}[theorem]{Question}

\newtheorem{construction}[theorem]{Construction}

\newtheorem*{remark}{Remark}

\usepackage{xcolor}
\usepackage[normalem]{ulem}

\begin{document}
\title[On subgroup perfect codes in vertex-transitive graphs]{On subgroup perfect codes in vertex-transitive graphs}

\author[Xia]{Binzhou Xia}
\address{School of Mathematics and Statistics\\The University of Melbourne\\Parkville, VIC 3010\\Australia}
\email{binzhoux@unimelb.edu.au}

\author[Zhang]{Junyang Zhang}
\address{School of Mathematical Sciences \& Chongqing Key Lab of Cognitive Intelligence and Intelligent Finance\\Chongqing Normal University\\Chongqing, 401331\\P. R. China}
\email{jyzhang@cqnu.edu.cn}

\author[Zhang]{Zhishuo Zhang}
\address{School of Mathematics and Statistics\\The University of Melbourne\\Parkville, VIC 3010\\Australia}
\email{zhishuoz@student.unimelb.edu.au}

\begin{abstract}
A subset $C$ of the vertex set $V$ of a graph $\Gamma$ is called a perfect code in $\Gamma$ if every vertex in $V\setminus C$ is adjacent to exactly one vertex in $C$. Given a group $G$ and a subgroup $H$ of $G$, a subgroup $A$ of $G$ containing $H$ is called a perfect code of the pair $(G,H)$ if there exists a coset graph $\Cos(G,H,U)$ such that the set of left cosets of $H$ in $A$ is a perfect code in $\Cos(G,H,U)$. In particular, $A$ is called a perfect code of $G$ if $A$ is a perfect code of the pair $(G,1)$. In this paper, we give a characterization of $A$ to be a perfect code of the pair $(G,H)$ under the assumption that $H$ is a perfect code of $G$. As a corollary, we derive an additional sufficient and necessary condition for $A$ to be a perfect code of $G$. Moreover, we establish conditions under which $A$ is not a perfect code of $(G,H)$, which is applied to construct infinitely many counterexamples to a question posed by Wang and Zhang [\emph{J.~Combin.~Theory~Ser.~A}, 196 (2023) 105737]. Furthermore, we initiate the study of determining which maximal subgroups of $S_n$ are perfect codes.
\end{abstract}

\maketitle

\section{Introduction}
Throughout this paper, all groups are finite and all graphs  are finite, undirected and simple. Let $\Gamma$ be a graph with vertex set $V$. A subset $C$ of $V$ is called a \emph{perfect code} in $\Gamma$ if every vertex in $V\setminus C$ is adjacent to exactly one vertex in $C$. A perfect code is also referred to as an \emph{efficient dominating set}~\cite{DS2003} or \emph{independent perfect dominating set}~\cite{L2001}.

The study of perfect codes in graphs can be traced back to at least 1973, when Biggs~\cite{B1973} extended the classical settings of error-correcting perfect codes, such as Hamming codes and Lee codes, to distance-transitive graphs. Among these studies, perfect codes in Cayley graphs have emerged as one of the most popular topics~\cite{CM2013,CWX2020,DSLW2017,FHZ2017,HBZ2018,ZZ2021,Z2019}. Let $e$ denote the identity element of any group. For a group $G$ and an inverse-closed subset $S$ of $G\setminus \{e\}$, the \emph{Cayley graph} $\Cay(G,S)$ is a graph with vertex set $G$ and edge set $\{\{g,sg\}\mid s\in S,\; g\in G\}$. A subset $C$ of the group $G$ is called a \emph{perfect code} of $G$ if $C$ is a perfect code in some Cayley graph of $G$.

Problems become particularly intriguing when we restrict perfect codes of $G$ to subgroups of $G$. Indeed, researchers often aim to find perfect codes with algebraic structures, such as group structure, as these not only provide deeper theoretical insights but also enable efficient storage and manipulation in computer systems. In 2018, the authors of~\cite{HBZ2018} found a sufficient and necessary condition for a normal subgroup $A$ of $G$ to be a perfect code:
\begin{equation}\label{eq:cond}
    \text{for each $x \in G$ such that $x^2 \in A$, there exists $y \in xA$ such that $y^2 = e$.}
\end{equation}
This condition led to the complete classification of subgroup perfect codes in cyclic groups, dihedral groups and generalized quaternion groups~\cite{HBZ2018,MWWZ2020}. Later,  the authors of~\cite{CWX2020} generalized condition~\eqref{eq:cond} to the following characterization of a general subgroup of $G$ to be a perfect code of $G$.

\begin{theorem}[{\cite[Theorem~1.2]{CWX2020}}]\label{thm:1}
    Let $G$ be a group and $A\leq G$. Then the following statements are equivalent:
    \begin{enumerate}[\rm(a)]
        \item \label{enu:pc1} $A$ is a perfect code of $G$;
        \item \label{enu:pc2} there exists an inverse-closed left transversal of $A$ in $G$;
        \item \label{enu:pc3} for each $x\in G$ such that $x^2\in A$ and $|A|/|A\cap A^x|$ is odd, there exists $y\in xA$ such that $y^2=e$;
        \item \label{enu:pc4} for each $x\in G$ such that $AxA=Ax^{-1}A$ and $|A|/|A \cap A^x|$ is odd, there exists $y\in xA$ such that $y^2=e$.
    \end{enumerate}
\end{theorem}

\begin{remark}
    Since most of the literature adopts the left coset convention, we present the ``left coset version" of~\cite[Theorem~1.2]{CWX2020} in Theorem~\ref{thm:1}. Moreover, we note that the equivalence of~\eqref{enu:pc1} and~\eqref{enu:pc4} is also proved in~\cite{ZZ2021}.
\end{remark}

Noting that distance-transitive graphs and Cayley graphs are special cases of vertex-transitive graphs, Wang and the second author~\cite{WZ2023} initiated the study of perfect codes in vertex-transitive graphs. A graph $\Gamma$ is said to be \emph{$G$-vertex-transitive} if $G$ is a group of automorphisms of $\Gamma$ that acts transitively on the vertex set of $\Gamma$. Since a graph is $G$-vertex-transitive if and only if it can be represented as a coset graph $\Cos(G,H,U)$ for some subgroup $H$ of $G$ and inverse-closed subset $U\subseteq G\setminus H$, they generalized the subgroup perfect codes of a group to the following: given a finite group $G$ and a subgroup $H$ of $G$, a subgroup $A$ of $G$ containing $H$ is called a \emph{subgroup perfect code of the pair $(G,H)$} if there exists a coset graph $\Cos(G,H,U)$ such that the set of left cosets of $H$ in $A$ is a perfect code in $\Cos(G,H,U)$. A sufficient and necessary condition for $A$ to be a perfect code of the pair $(G,H)$ is given in~\cite{WZ2023} as follows.

\begin{theorem}[{\cite[Theorem 3.1]{WZ2023}}]\label{thm:pc}
    Let $G$ be a group and $H\leq A\leq G$. Then $A$ is a perfect code of $(G,H)$ if and only if there exists a left transversal $X$ of $A$ in $G$ such that $XH=HX^{-1}$.
\end{theorem}

Since $A$ is a perfect code of the pair $(G,1)$ if and only if $A$ is a perfect code of $G$, studying perfect codes of the group $G$ is equivalent to studying the special case of perfect codes of the pair $(G,H)$ when $H=1$. Hence, it is natural to consider the generalization of Theorem~\ref{thm:1} to a characterization of perfect codes of the pair $(G,H)$. The following attempt in~\cite{WZ2023} generalizes the implication of~Theorem~\ref{thm:1}\eqref{enu:pc4} from~Theorem~\ref{thm:1}\eqref{enu:pc1}.

\begin{theorem}[{\cite[Theorem 3.6]{WZ2023}}]\label{thm:1.2}
    Let $G$ be a group and $H\leq A\leq G$. If $A$ is a perfect code of $(G,H)$, then for each $g\in G$, either $|A\{g,g^{-1}\}A|/|A|$ is even or the left coset $gA$ contains an element $x$ such that $x^2\in H^y$ for some $y\in A$.
\end{theorem}

Inspired by this, the following open problem is posed in~\cite{WZ2023} regarding a potential characterization of perfect codes of the pair $(G,H)$.

\begin{question}[{\cite[Problem~5.1]{WZ2023}}]\label{ques:1}
    Does the converse of Theorem~\ref{thm:1.2} hold?
\end{question}

In this paper, we provide the following characterization of a subgroup $A$ of $G$ containing a subgroup $H\leq A$ to be a perfect code of the pair $(G,H)$ under the assumption that $H$ is a perfect code of $G$.

\begin{theorem}\label{thm:pair}
    Let $G$ be a group and $H\leq A\leq G$. Suppose that $H$ is a perfect code of $G$. Then the following statements are equivalent:
    \begin{enumerate}[\rm(a)]
        \item $A$ is a perfect code of the pair $(G,H)$;
        \item there exists an inverse-closed left transversal $X$ of $A$ in $G$ such that $XH=HX$;
        \item for each $g\in G$, there exists an inverse-closed left transversal $Y$ of $A$ in $A\{g,g^{-1}\}A$ such that $Y$ is a left transversal of $H$ in $HYH$.
    \end{enumerate}
\end{theorem}

For the special case when $H=1$, the assumption that $H$ is a perfect code of $G$ is trivially satisfied. Therefore, taking $H=1$, the above theorem gives the following characterization of a subgroup $A$ of $G$ to be a perfect code of $G$, including an additional sufficient and necessary condition (Corollary~\ref{cor:new}\eqref{enu:newpc3}) that is not found in the existing literature.

\begin{corollary}\label{cor:new}
    Let $G$ be a group and $A\leq G$. Then the following are equivalent:
    \begin{enumerate}[\rm(a)]
        \item \label{enu:newpc1} $A$ is a perfect code of $G$;
        \item \label{enu:newpc2} there exists an inverse-closed left transversal of $A$ in $G$;
        \item \label{enu:newpc3} for each $g\in G$,  there is an inverse-closed left transversal $Y$ of $A$ in $A\{g,g^{-1}\}A$.
    \end{enumerate}
\end{corollary}

Furthermore, using Theorem~\ref{thm:pair}, we deduce the following result, which provides conditions under which $A$ is not a perfect code of $(G,H)$. As standard notation in group theory, for a subgroup $H$ of a group $G$, the normal closure of $H$ in $G$ (the smallest normal subgroup of $G$ containing $H$) is denoted by $H^G$, and the normalizer of $H$ in $G$ is denoted by $\mathrm{N}_G(H)$.

\begin{proposition}\label{prop:closure}
    Let $G$ be a group and $H\leq A\leq G$ such that $H$ is nonnormal in $G$. Suppose that $H$ is a perfect code of $G$. If $H^G\leq A\leq \mathrm{N}_{G}(H)$, then $A$ is not a perfect code of $(G,H)$.
\end{proposition}

It is worth noting that Proposition~\ref{prop:closure} is the first result other than the contrapositive of Theorem~\ref{thm:1.2} to establish that some subgroup is not a perfect code of a pair $(G,H)$. We then construct three infinite families of counterexamples (Constructions~\ref{const:1},~\ref{const:2} and \ref{const:3}) to answer Question~\ref{ques:1} in the negative, where Proposition~\ref{prop:closure} is utilized (see the proofs of Propositions~\ref{prop:counterexample2}\eqref{enu:c2b} and~\ref{prop:counterexample3}\eqref{enu:c3b}) to validate that the families of triples $(G,A,H)$ in Constructions~\ref{const:2} and~\ref{const:3} are indeed counterexamples. The counterexamples in Construction~\ref{const:3} also illustrate that \cite[Theorem 3.6]{ZZ2021} cannot be generalized to group pairs.

In~\cite{WZ2023}, to provide examples after introducing the new concept of perfect codes of the pair $(G,H)$, it is proved~\cite[Proposition~5.6]{WZ2023} that $S_{n-1}$ is a perfect code of $(S_n,S_3)$ for each $n\geq 5$. In the following proposition, we extend this result significantly, thereby greatly expanding the range of known examples.

\begin{proposition}\label{prop:3}
    For each positive integers $\ell<m<n$, the group $S_m$ is a perfect code of $(S_n,S_\ell)$, where $S_\ell<S_m<S_n$ in the natural embedding.
\end{proposition}

As a special case when $\ell=1$, we obtain that $S_m$ is perfect code of $S_n$ for each $1\leq m<n$. This result motivates us to further investigate perfect codes of $S_n$. Clearly, studying all subgroups of $S_n$ is not practical, as every finite group is a subgroup of some symmetric group. Nevertheless, maximal subgroups of $S_n$ form a broad and well-studied family of subgroups of $S_n$ (see~\cite{LPS1987} for example), and in coding theory, larger subgroup perfect codes correspond to sparser underlying Cayley graphs, making errors less likely to occur. Therefore, we focus on determining which maximal subgroups of $S_n$ are perfect codes of $S_n$. Our computational results by GAP \cite{GAP4} show that every maximal subgroup of $S_n$, where $n\leq 9$, is a perfect code of $S_n$. This motivates us to ask the question below.

\begin{question}\label{prob:2}
    Does there exist any maximal subgroup $A$ of $S_n$ such that $A$ is not a perfect code of $S_n$?
\end{question}

We answer the question for the case when $A$ is intransitive (see Proposition~\ref{prop:intransitive}) and establish a partial result on affine groups $A$ (see Proposition~\ref{prop:affine}). These results, along with our computational experiments, suggest that the answer to Question~\ref{prob:2} may be negative, at least for most cases. However, further investigation is required to determine whether exceptions exist and to fully characterize the conditions, if any, under which maximal subgroups of $\Sy_n$ fail to be perfect codes.

The rest of this paper is organized as follows. In Section~\ref{sec:2}, we prove Theorem~\ref{thm:pair} and finally end the section by proving Proposition~\ref{prop:closure}. In the next section, we construct three infinite families of counterexamples to answer Question~\ref{ques:1} negatively. In Section~\ref{sec:4}, we prove Proposition~\ref{prop:3} and present our results on Question~\ref{prob:2}.

\section{Characterization of subgroup perfect codes}\label{sec:2}

In this section, we prove Theorem~\ref{thm:pair} and Proposition~\ref{prop:closure}. For Theorem~\ref{thm:pair}, we first prove the equivalence between~\eqref{enu:newpc1} and~\eqref{enu:newpc3}, and then prove the equivalence between~\eqref{enu:newpc1} and~\eqref{enu:newpc2}.

\begin{proof}[Proof of Theorem~$\ref{thm:pair}$]
    \eqref{enu:newpc3}$\Rightarrow$\eqref{enu:newpc1}:
    Assume that for each $g\in G$ there is an inverse-closed left transversal $Y$ of $A$ in $A\{g,g^{-1}\}A$ such that $Y$ is a left transversal of $H$ in $HYH$. Note that $G=\bigsqcup _{i=1}^m A\{g_i,g_i^{-1}\}A$ for some $g_1,\ldots, g_m\in G$. Then for each $i\in\{1,\ldots,m\}$, there is an inverse-closed left transversal $Y_i$ of $A$ in $A\{g_i,g_i^{-1}\}A$ such that $Y_i$ is a left transversal of $H$ in $HY_iH$. Let $X=\bigsqcup_{i=1}^m Y_i$. Clearly, $X$ is an inverse-closed left transversal of $A$ in $G$. By Theorem~\ref{thm:pc}, we are left to prove $XH=HX^{-1}$. Take an arbitrary $hx^{-1}\in HX^{-1}$, where $h\in H$ and $x\in X$, and assume that $x\in Y_i$ for some $i\in\{1,\ldots, m\}$. Then $hx^{-1}\in HY_i^{-1}H=HY_iH$. Since $Y_i$ is a left transversal of $H$ in $HY_iH$, there exists $y\in Y_i$ such that $hx^{-1}\in yH\subseteq XH$. Therefore, $HX^{-1}\subseteq XH$. This together with $|HX^{-1}|=|(XH)^{-1}|=|XH|$ implies that $HX^{-1}=XH$, completing the proof of \eqref{enu:newpc3}$\Rightarrow$\eqref{enu:newpc1}.

    \eqref{enu:newpc1}$\Rightarrow$\eqref{enu:newpc3}: Assume that $A$ is a perfect code of $(G,H)$.  It follows from Theorem~\ref{thm:pc} that there exists a left transversal $T$ of $A$ in $G$ such that $TH=HT^{-1}$. Take an arbitrary $g\in G$. Since $A\{g,g^{-1}\}A$ is a union of disjoint left cosets of $A$, there exists a unique subset $X\subseteq T$ such that $A\{g,g^{-1}\}A=XA$. Noting that $XH=XA\cap TH$ and $(XA)^{-1}=(A\{g,g^{-1}\}A)^{-1}=A\{g,g^{-1}\}A=XA$, we have
    \[
    HX^{-1}=(XH)^{-1}=(XA\cap TH)^{-1}=(XA)^{-1}\cap HT^{-1}=XA\cap TH=XH.
    \]
    As a consequence,
    \begin{equation}\label{eq:XH=HXH}
        HXH=H(HX^{-1})=HX^{-1}=XH=(XH)H=HX^{-1}H.
    \end{equation}

    Let $\Gamma=(V,E)$ be the graph such that $V$ is the set of left cosets of $H$ in $G$, and $\{xH, yH\}\in E$ if and only if $y\in Hx^{-1}H$. In other words, two distinct vertices $xH$ and $yH$ are adjacent if and only if $HyH=Hx^{-1}H$. In particular, $xH$ is adjacent to $x^{-1}H$. Therefore, the subgraph $\Gamma_x$ induced by $H\{x,x^{-1}\}H$ is a connected component of $\Gamma$. Moreover, $\Gamma_x$ is a complete graph of order $|H|/|H\cap H^x|$ if $HxH= Hx^{-1}H$, and is a complete bipartite graph if $HxH\neq Hx^{-1}H$. Denote by $\Gamma_{X}$ the subgraph induced by $HXH$. By~\eqref{eq:XH=HXH}, there exist $x_1,\ldots,x_k\in X$ such that $HXH=\bigsqcup_{i=1}^k H\{x_i,x_i^{-1}\}H$. Hence, the connected components of $\Gamma_X$ are $\Gamma_{x_1},\ldots, \Gamma_{x_k}$. Let
    \[
        I=\{i\in \{1,\ldots, k\}\mid Hx_iH= Hx_i^{-1}H\; \text{and}\; |H|/|H\cap H^{x_i}|\; \text{is odd}\}.
    \]
    Denote by $\Gamma_0$ the subgraph of $\Gamma_X$ induced by the vertex set $\bigsqcup_{i\in \{1,\ldots,k\}\setminus I} H\{x_i,x_i^{-1}\}H$.
    Then each connected component of $\Gamma_0$ is either a complete bipartite graph or a complete graph of even order. Therefore, $\Gamma_0$ has a matching, say, $\{u_1H, v_1H\},\ldots, \{u_\ell H, v_\ell H\}$. According to the adjacency in $\Gamma$, for each $j\in\{1,\ldots,\ell\}$, there exist $a_j,b_j\in H$ such that $v_j=a_ju_j^{-1}b_j$. Let $z_j=v_jb_j^{-1}$. Then $z_jH=v_jH$ and $z_j^{-1}H=b_jv_j^{-1}H=u_ja_j^{-1}H=u_jH$. Hence,
    \begin{equation}\label{eq:vertices}
        \{z_1H,z_1^{-1}H,\ldots,z_\ell H,z_\ell^{-1} H\}=\{u_1H,v_1H,\ldots,u_\ell H,v_\ell H\}=\bigsqcup_{i\in \{1,\ldots,k\}\setminus I} H\{x_i,x_i^{-1}\}H.
    \end{equation}
    Since $H$ is a perfect code of $G$, it follows from Theorem~\ref{thm:1} that, for each $i \in I$, there exists $y_i\in x_iH$ with $y_i^2=e$. Let
    \[
        Y=\{z_1,z_1^{-1},\ldots,z_\ell,z_\ell^{-1}\}\cup \{y_i\mid i\in I\}.
    \]
    Then the union of all the vertices of $\Gamma_X$ is
    \[
        HXH=\Big(\bigsqcup_{i\in \{1,\ldots,k\}\setminus I} H\{x_i,x_i^{-1}\}H\Big)\cup \Big(\bigsqcup_{i\in I} H\{x_i,x_i^{-1}\}H\Big)=YH.
    \]

    To complete the proof of \eqref{enu:newpc1}$\Rightarrow$\eqref{enu:newpc3}, we show that $Y$ is an inverse-closed left transversal of both $H$ in $HYH$ and $A$ in $A\{g,g^{-1}\}A$. Clearly, $Y$ is inverse-closed. Take any two distinct elements $y$ and $z$ in $Y$. Since $yH$ and $zH$ are distinct vertices of $\Gamma_X$, we have $yH\cap zH=\emptyset$. This shows that $Y$ is a left transversal of $H$ in $YH$. From $HXH=YH$ we deduce that $HYH=H(HXH)=HXH=YH$. Hence, $Y$ is a left transversal of $H$ in $HYH$. Moreover, we derive from~\eqref{eq:XH=HXH} that $YH=HXH=XH$, and so there exist distinct $u$ and $v$ in $X$ such that $yH=uH$ and $zH=vH$. Since $X$ is a subset of the left transversal $T$ of $A$ such that $A\{g,g^{-1}\}A=XA$, it follows that $yA\cap zA=uA\cap vA=\emptyset$ and $YA=YHA=XHA=XA=A\{g,g^{-1}\}A$. Therefore, $Y$ is also a left transversal of $A$ in $A\{g,g^{-1}\}A$, as desired.

    \eqref{enu:newpc2}$\Rightarrow$\eqref{enu:newpc1}:
   This is a direct corollary of Theorem~\ref{thm:pc}.

    \eqref{enu:newpc1}$\Rightarrow$\eqref{enu:newpc2}:
    Assume that $A$ is a perfect code of $(G,H)$ and write $G=\bigsqcup _{i=1}^m A\{g_i,g_i^{-1}\}A$ for some $g_1,\ldots, g_m\in G$. By \eqref{enu:newpc1}$\Leftrightarrow$\eqref{enu:newpc3}, for each $i\in\{1,\ldots,m\}$, there exists an inverse-closed left transversal $Y_i$ of $A$ in $A\{g_i,g_i^{-1}\}A$ such that $Y_i$ is a left transversal of $H$ in $HY_iH$. Hence, $HY_i=HY_i^{-1}=(Y_iH)^{-1}=(HY_iH)^{-1}=HY_i^{-1}H=HY_iH=Y_iH$.
        Let $X=\bigsqcup_{i=1}^m Y_i$. Then $X$ is an inverse-closed left transversal of $A$ in $\bigsqcup _{i=1}^m A\{g_i,g_i^{-1}\}A=G$. Moreover, $XH=\bigsqcup _{i=1}^mY_iH=\bigsqcup _{i=1}^mHY_i=HX$. This completes the proof.
\end{proof}

\begin{proof}[Proof of Proposition~$\ref{prop:closure}$]
    Suppose for a contradiction that $A$ is a perfect code of $(G,H)$. By Theorem~\ref{thm:pair}, there exists an inverse-closed left transversal $X$ of $A$ in $G$ such that $XH=HX$. Since $H$ is not normal in $G$ but normal in $A$, there exist $x\in X$ and $h\in H$ such that $xhx^{-1}\notin H$. Since $xH\subseteq XH=HX$, there exist $x_1,x_2\in X$ and $h_1,h_2\in H$ such that $x=h_1x_1$ and $xh=h_2x_2$. Then $h_1x_1A=xA=xhA=h_2x_2A$. Note from $H^G\leq A$ that
    \[
        h_1x_1A=x_1h_1^{x_1}A=x_1A \ \text{ and }\ h_2x_2A=x_2h_2^{x_2}A=x_2A.
    \]
    Thereby we obtain $x_1A=x_2A$, which implies $x_1=x_2$. Hence, $xhx^{-1}=h_2x_2(h_1x_1)^{-1}=h_2h_1^{-1}$, contradicting  $xhx^{-1}\notin H$.
\end{proof}

\section{Counterexamples to Question~\ref{ques:1}}\label{sec:3}

In this section, we construct three infinite families of triples $(G,A,H)$ with $H\leq A\leq G$, and show that they are counterexamples of the converse of Theorem~\ref{thm:1.2}, which leads to a negative answer to Question~\ref{ques:1}.
Recall that the statement of Theorem~\ref{thm:1.2} is the implication of $A$ being a perfect code of $(G,H)$ to one of the following for each $g\in G$:
 \begin{enumerate}[\rm(i)]
        \item \label{enu:cond1} the left coset $gA$ contains an element $x$ such that $x^2\in H^y$ for some $y\in A$.
        \item \label{enu:cond2} $|A\{g,g^{-1}\}A|/|A|$ is even;
    \end{enumerate}

\begin{construction}\label{const:1}
    Let $G=D_{8n}=\langle a,b\mid a^{4n}=b^2=1,\; bab=a^{-1}\rangle$ for some positive integer $n$, and take $A=\langle a^{2n}, b\rangle\cong C_2\times C_2$ and $H=\langle b\rangle\cong C_2$.
\end{construction}

As asserted by the proposition below, the triples $(G,A,H)$ in Construction~\ref{const:1} satisfy~\eqref{enu:cond1} for each $g\in G$ but not that $A$ is a perfect code of $(G,H)$.

\begin{proposition}\label{prop:counterexample1}
    For each triple $(G,A,H)$ in Construction~$\ref{const:1}$, the following statements hold:
    \begin{enumerate}[\rm(a)]
        \item \label{enu:c1a} for each $g\in G$, the left coset $gA$ contains an element $x$ such that $x^2\in H^y$ for some $y\in A$;
        \item \label{enu:c1b} $A$ is not a perfect code of $(G,H)$.
    \end{enumerate}
\end{proposition}

\begin{proof}
Clearly, $A, aA,\ldots, a^{2n-1}A$ are all the left cosets of $A$ in $G$. Since $(a^ib)^2=1\in H$ for each $i\in\{0,\ldots,2n-1\}$, we obtain that each left coset of $A$ in $G$ contains an involution. This completes the proof of statement~\eqref{enu:c1a}.

Let $X$ be an arbitrary left transversal of $A$ in $G$. By Theorem~\ref{thm:pc}, to prove statement~\eqref{enu:c1b}, it suffices to show that $XH\neq HX^{-1}$. Write
\[
X=\{x_0,ax_1,\ldots,a^{2n-1}x_{2n-1}\},
\]
where $x_i\in A$ for $i\in\{0,\ldots,2n-1\}$. Since $A=\langle a^{2n}, b\rangle$ and $H=\langle b\rangle$, we have $x_i\in a^{2n\gamma_i}H$ with $\gamma_i\in \{0,1\}$ for each $i\in\{0,\ldots,2n-1\}$. It follows that
\begin{equation*}
    HX^{-1}=(XH)^{-1}=\bigsqcup_{i=0}^{2n-1}\{a^{-(i+2n\gamma_i)},ba^{-(i+2n\gamma_i)}\}
    =\bigsqcup_{i=0}^{2n-1}\{a^{-(i+2n\gamma_i)},a^{i+2n\gamma_i}b\}.
\end{equation*}
Suppose for a contradiction that $XH=HX^{-1}$. Then
\[
    \bigsqcup_{i=0}^{2n-1}\{a^{i+2n\gamma_i},a^{i+2n\gamma_i}b\}=\bigsqcup_{i=0}^{2n-1}\{a^{-(i+2n\gamma_i)},a^{i+2n\gamma_i}b\},
\]
which implies that
\[
    \big\{a^{i+2n\gamma_i}\mid i\in\{0,\ldots,2n-1\}\big\}=\big\{a^{-(i+2n\gamma_i)}\mid i\in\{0,\ldots,2n-1\}\big\}.
\]
Intersecting both sides with $a^nA$, we obtain
\[
    a^{n+2n\gamma_n}=a^{-(n+2n\gamma_{n})}.
\]
This implies that $a^{2n}=1$, a contradiction. Thus, $XH\neq HX^{-1}$, as desired.
\end{proof}

Before presenting the second and third infinite families of counterexamples, let us fix some notation for the rest of this section. For the finite field $\FF_{q}$ where $q=p^f$ for some prime $p$ and positive integer $f$, let $\FF_{q}^+$ and $\FF_{q}^\times$ denote its additive group and multiplicative group respectively. As $\FF_{q}^\times$ is cyclic, we take a generator $\omega$ of $\FF_{q}^\times$. For $u\in\FF_{q}$, let $R(u)$ be the permutation on $\FF_{q}$ sending $v$ to $v+u$ for each $v\in\FF_{q}$. Then $R$ gives an embedding of the group $\FF_{q}^+$ into $\Sym(\FF_{q})$. Let $e=R(0)$ be the identity in $\Sym(\FF_{q})$, and let $s$ and $\varphi$ be permutations on $\FF_q$ such that
\begin{equation}\label{eq:s}
s\colon v\mapsto v\omega\ \text{ and }\ \varphi\colon v\mapsto v^p
\end{equation}
for each $v\in\bbF_q$. Then $s$ is a Singer cycle and $\varphi$ is the Frobenius automorphism of the field $\FF_q$.
The following lemma will be used repeatedly without reference.

\begin{lemma}\label{lm:3.3}
    In the above notation, the following statements hold:
    \begin{enumerate}[\rm(a)]
        \item \label{enu:a} $s$ and $\varphi$ have orders $q-1$ and $f$, respectively, such that $s^\varphi=s^p$;
        \item \label{enu:b} $R(u)^s=R(u\omega)$ and $R(u)^\varphi=R(u^p)$ for each $u\in\bbF_q$.
    \end{enumerate}
\end{lemma}

\begin{proof}
    The orders of $s$ and $\varphi$ are readily seen from~\eqref{eq:s}. For each $u$ and $v$ in $\FF_{q}$,
\begin{align*}
    v^{s^\varphi}&=v^{\varphi^{-1}s\varphi}=(v^{\varphi^{-1}}\omega)^\varphi
    =v^{\varphi^{-1}\varphi}\omega^\varphi=v\omega^p=v^{s^p},\\
    v^{R(u)^s}&=v^{s^{-1}R(u)s}=(v\omega^{-1}+u)\omega=v+u\omega
    =v^{R(u\omega)},\\
    v^{R(u)^\varphi}&=v^{\varphi^{-1}R(u)\varphi}=(v^{\varphi^{-1}}+u)^\varphi
    =v^{\varphi^{-1}\varphi}+u^\varphi=v^{R(u^p)}.
\end{align*}
This verifies the lemma.
\end{proof}

Our second infinite family of counterexamples turn out to satisfy~\eqref{enu:cond2} for each $g\in G\setminus A$. It is worth noting that, if $g\in A$, then condition~\eqref{enu:cond1} holds trivially, whereas condition~\eqref{enu:cond2} fails to hold.

\begin{construction}\label{const:2}
    Let $f=2d$ for some integer $d\geq2$, let $V=R(\FF_{2^f})\leq \Sym(\FF_{4^d})$, and let $s$ and $\varphi$ be permutations on $\bbF_{4^d}$ defined by~\eqref{eq:s} with $p=2$. Take $G=V\rtimes(\langle s^{2^d-1}\rangle\rtimes\langle\varphi^d\rangle)$, $A=V\rtimes \langle \varphi^d \rangle$ and $H=\langle R(\omega),R(\omega^{2^d})\rangle$.
\end{construction}

Note in Construction~\ref{const:2} that $\varphi^d$ is an involution inverting $s^{2^d-1}$, and so $G\cong C_2^{2d}\rtimes D_{2(2^d+1)}$.
Moreover, $A\cong C_2^{2d}\rtimes C_2$, while $H=\langle R(\omega),R(\omega^{2^d})\rangle\cong C_2\times C_2$ consists of $e$, $R(\omega)$, $R(\omega^{2^d})$ and $R(\omega+\omega^{2^d})$. Since
\begin{equation}\label{eq:0}
   R(\omega)^{\varphi^d}=R(\omega^{2^d}),\ \  R(\omega^{2^d})^{\varphi^d}=R(\omega),\ \   R(\omega+\omega^{2^d})^{\varphi^d}=R(\omega+\omega^{2^d}),
\end{equation}
we see that $H$ is normalized by $\varphi^d$. This in conjunction with the fact that $H$ is a subgroup of the abelian group $V$ indicates that $H$ is normalized by $V\rtimes \langle \varphi^d \rangle=A$; that is,
\begin{equation}\label{eq:normal}
A\leq\mathrm{N}_G(H).
\end{equation}

\begin{proposition}\label{prop:counterexample2}
    For each triple $(G,A,H)$ in Construction~$\ref{const:2}$, the following statements hold:
    \begin{enumerate}[\rm(a)]
        \item \label{enu:c2a} $|A\{g,g^{-1}\}A|/|A|$ is even for each $g\in G\setminus A$;
        \item \label{enu:c2b} $A$ is not a perfect code of $(G,H)$.
    \end{enumerate}
\end{proposition}

\begin{proof}
Take an arbitrary $g\in G\setminus A$. Noting that the set of right cosets of $A$ in $G$ consists of $As^{i(2^d-1)}$ with $i\in \{0,1,\ldots,2^d\}$, we have $g\in As^{k(2^d-1)}$ for some $k\in\{1,\ldots,2^d\}$. It follows that $A^g=A^{s^{k(2^d-1)}}$. Suppose that $R(u)\varphi^{di}=(R(v)\varphi^{dj})^{s^{k(2^d-1)}}\in A\cap A^g$ for some $u,v\in\bbF_{4^d}$ and $i,j\in \{0,1\}$. Since $V\trianglelefteq G$ and $V\cap \langle\varphi\rangle=1$, we then derive that $R(u)=R(v)^{s^k}$ and $\varphi^{di}=(\varphi^{dj})^{s^{k(2^d-1)}}$, whence
\begin{align*}
    \varphi^{d(i-j)}&=\varphi^{-dj}s^{-k(2^d-1)}\varphi^{dj}s^{k(2^d-1)}\\
    &=(s^{-k(2^d-1)})^{\varphi^{dj}}s^{k(2^d-1)}=s^{-k(2^d-1)2^{dj}}s^{k(2^d-1)}=s^{k(2^d-1)(-2^{dj}+1)}.
\end{align*}
Then it follows from $\langle \varphi\rangle \cap \langle s\rangle=1$ that $i=j$ and $4^d-1$ divides $k(2^d-1)(2^{dj}-1)$. Thus, $i=j=0$, which implies that $A\cap A^g\leq V$. Clearly, $A\cap A^g\geq V\cap V^g=V$, as $V\trianglelefteq G$. Therefore, we obtain $A\cap A^g=V$ and hence
\[
    \frac{|AgA|}{|A|}=\frac{|A|}{|A\cap A^g|}=\frac{|A|}{|V|}=|\langle \varphi^d\rangle|=2.
\]
Since $|AgA|/|A|$ divides $|A\{g,g^{-1}\}A|/|A|$, we conclude that $|A\{g,g^{-1}\}A|/|A|$ is even, as statement~\eqref{enu:c2a} asserts.

For statement~\eqref{enu:c2b}, in order to apply Proposition~\ref{prop:closure}, we first show that $H$ is a perfect code of $G$. Take an arbitrary $x\in G$ and write $x=R(v)s^{i(2^d-1)}\varphi^{dj}$ for some $v\in \bbF_{4^d}$, $i \in \{0,1,\ldots,2^d\}$ and $j\in\{0,1\}$. Assume that $x^2\in H$ and $|H|/|H\cap H^x|$ is odd. By Theorem~\ref{thm:1}, it suffices to show that there exists $z\in xH$ such that $z^2=e$. For convenience, write
\[
    \ell=i(2^d-1).
\]
Observe that
\begin{align}
    x^2
      =R(v)s^\ell\varphi^{dj}R(v)s^\ell\varphi^{dj}
      &=R(v)R(v)^{(s^\ell\varphi^{dj})^{-1}}s^\ell\varphi^{dj}s^\ell\varphi^{dj}\nonumber\\
      &=R(v)R((v^{2^{dj}}\omega^{-\ell}))s^\ell(s^\ell)^{\varphi^{dj}}
      =R(v+v^{2^{dj}}\omega^{-\ell})s^{\ell(2^{dj}+1)}.\label{eq:x^2}
  \end{align}
We derive from $x^2\in H$ that $s^{\ell(2^{dj}+1)}=e$. If $j=0$, then this implies $\ell=0$, and it follows that $x^2=R(v+v)=e$, as desired. Now we assume $j=1$. Then
\begin{align*}
    H^x
    &=\{e, R(\omega), R(\omega^{2^d}),R(\omega+\omega^{2^d})\}^{R(v)s^\ell\varphi^d}\\
    &=\{e, R(\omega^{1+\ell}), R(\omega^{2^d+\ell}),R(\omega^{1+\ell}+\omega^{2^d+\ell})\}^{\varphi^d}\\
    &=\{e, R(\omega^{(1+\ell)2^d}), R(\omega^{(2^d+\ell)2^d}), R(\omega^{(1+\ell)2^d}+\omega^{(2^d+\ell)2^d})\}.
\end{align*}
Since $|H|=4$ and $|H|/|H\cap H^x|$ is odd, we have $H^x=H$, which means
\begin{equation*}
\{e, R(\omega^{(1+\ell)2^d}), R(\omega^{(2^d+\ell)2^d}), R(\omega^{(1+\ell)2^d}+\omega^{(2^d+\ell)2^d})\}=\{e, R(\omega), R(\omega^{2^d}), R(\omega+\omega^{2^d})\}.
\end{equation*}
The conjugation of $\varphi^d$ on both sides of this equation gives
\[
    \{e, R(\omega^{1+\ell}), R(\omega^{2^d+\ell}), R(\omega^{1+\ell}+\omega^{2^d+\ell})\}=\{e, R(\omega), R(\omega^{2^d}), R(\omega+\omega^{2^d})\},
\]
which implies that
\[
    \{\omega^{1+\ell}, \omega^{2^d+\ell}, \omega^{1+\ell}+\omega^{2^d+\ell}\}=\{\omega, \omega^{2^d}, \omega+\omega^{2^d}\}.
\]
Suppose that $\ell\neq 0$. Then $\omega^{1+\ell}\ne \omega$, $\omega^{2^d+\ell}\ne \omega^{2^d}$ and
$\omega^{1+\ell}+\omega^{2^d+\ell}\ne\omega+\omega^{2^d}$. Accordingly, one of the following equations holds:
\begin{align}
    &(\omega^{1+\ell}, \omega^{2^d+\ell}, \omega^{1+\ell}+\omega^{2^d+\ell})=(\omega^{2^d}, \omega+\omega^{2^d}, \omega),\label{eq:1}\\
    &(\omega^{1+\ell}, \omega^{2^d+\ell}, \omega^{1+\ell}+\omega^{2^d+\ell})=(\omega+\omega^{2^d}, \omega, \omega^{2^d}).\label{eq:2}
\end{align}
For~\eqref{eq:1}, we have $\omega^{2^d+\ell}=\omega+\omega^{2^d}=(\omega+\omega^{2^d})^{2^d}=(\omega^{2^d+\ell})^{2^d}$ and $\omega^{2^d+\ell}=\omega^{1+\ell}\omega^{2^d-1}=\omega^{2^d}\omega^{2^d-1}$, from which it follows that
\[
    1=(\omega^{2\cdot2^d-1})^{2^d}(\omega^{2^d}\omega^{2^d-1})^{-1}=\omega^{2\cdot4^d-2^d+1-2\cdot2^d}=(\omega^{-1})^{3\cdot2^d-3}.
\]
For~\eqref{eq:2}, we have $\omega^{1+\ell}=\omega+\omega^{2^d}=(\omega+\omega^{2^d})^{2^d}=\omega^{(1+\ell)2^d}$ and $\omega^{1+\ell}=\omega^{2^d+\ell}\omega^{1-2^d}=\omega\omega^{1-2^d}$, whence
\[
    1=(\omega^{2-2^d})^{2^d}(\omega\omega^{1-2^d})^{-1}=\omega^{2\cdot2^d-4^d-2+2^d}=\omega^{3\cdot2^d-3}.
\]
Thus, both~\eqref{eq:1} and~\eqref{eq:2} imply that $4^d-1$ divides $3\cdot2^d-3$, which is impossible as $d\geq2$. Consequently, $\ell=0$. Then $x=R(v)\varphi^d$, and this together with $j=1$ leads to $x^2=R(v+v^{2^d})$ by~\eqref{eq:x^2}.
Since $x^2\in H$ and $R(v+v^{2^d})$ is fixed by the conjugation of $\varphi^d$, we conclude from~\eqref{eq:0} that $x^2\in \{e,R(\omega+\omega^{2^d})\}$. Moreover, if $x^2=R(\omega+\omega^{2^d})$, then for the element $z=xR(\omega)\in xH$, we have
\[
    z^2=(xR(\omega))^2=x^2R(\omega)^{x}R(\omega)
    =x^2R(\omega)^{R(v)\varphi^d}R(\omega)
    =R(\omega+\omega^{2^d})R(\omega+\omega^{2^d})=e.
\]
Therefore, there always exists $z\in xH$ such that $z^2=e$. This proves that $H$ is a perfect code of $G$.

Next, we show that $H$ is a nonnormal subgroup of $G$ by proving that $H^{s^{2^d-1}}\neq H$. It is clear that
\[
    H^{s^{2^d-1}}=\{e, R(\omega^{2^d}), R(\omega^{2^{d+1}-1}), R(\omega^{2^d}+\omega^{2^{d+1}-1})\}.
\]
Suppose for a contradiction that $H^{s^{2^d-1}}=H$. Then by~\eqref{eq:0}, the conjugation of $\varphi^d$ fixes one of $R(\omega^{2^d})$, $R(\omega^{2^{d+1}-1})$ or $R(\omega^{2^d}+\omega^{2^{d+1}-1})$ and interchanges the other two. Since
\[
    R(\omega^{2^d})^{\varphi^d}=R(\omega)\neq R(\omega^{2^d})\ \text{ and }\ R(\omega^{2^{d+1}-1})^{\varphi^d}=R(\omega^{2-2^d})\neq R(\omega^{2^{d+1}-1}),
\]
it follows that $\varphi^d$ interchanges $R(\omega^{2^d})$ and $R(\omega^{2^{d+1}-1})$. However,
\[
    R(\omega^{2^d})^{\varphi^d}=R(\omega)\neq R(\omega^{2^{d+1}-1}),
\]
a contradiction. Therefore, $H$ is a nonnormal subgroup of $G$.

Finally, by Proposition~\ref{prop:closure}, it remains to show that $H^G\leq A\leq \mathrm{N}_G(H)$. The fact $A\leq \mathrm{N}_G(H)$ is established in~\eqref{eq:normal}. Since $H\leq V$ and $V\trianglelefteq G$, we infer that $H^G\leq V^G=V\leq A$. This completes the proof of statement~\eqref{enu:c2b}.
\end{proof}

In our last infinite family of counterexamples, neither~\eqref{enu:cond1} is satisfied for all $g\in G$, nor~\eqref{enu:cond2} is satisfied for all $g\in G\setminus A$. (Recall that every $g$ in $A$ satisfies~\eqref{enu:cond1} and does not satisfy~\eqref{enu:cond2}.)

\begin{construction}\label{const:3}
    Let $q=p^f$ such that both the prime $p$ and positive integer $f$ are odd, let $V=R(\FF_q)\leq \Sym(\FF_q)$, and let $s$ and $\varphi$ be permutations on $\bbF_q$ defined by~\eqref{eq:s}. Take $G=V\rtimes(\langle s\rangle\rtimes \langle \varphi\rangle)$, $A=V\rtimes \langle \varphi \rangle$ and $H=\langle R(1)\rangle$.
\end{construction}

In Construction~\ref{const:3}, the group $G$ is $\AGaL(1,q)$, the group $A$ has odd order, and $H\cong C_p$. Recall by~\cite[Theorem 3.6]{ZZ2021} that every subgroup of $G$ of odd order must be a perfect code of $G$. We will show that $A$ is not a perfect of $(G,H)$, which indicates that \cite[Theorem 3.6]{ZZ2021} cannot be generalized to group pairs.

\begin{proposition}\label{prop:counterexample3}
    For each triple $(G,A,H)$ in Construction~$\ref{const:3}$, the following statements hold:
    \begin{enumerate}[\rm(a)]
        \item \label{enu:c3a} for each $g\in G\setminus A$, either $gA$ contains an element $x$ such that $x^2\in H^y$ for some $y\in A$ or $|A\{g,g^{-1}\}A|/|A|$ is even;
        \item \label{enu:c3b} $A$ is not a perfect code of $(G,H)$.
    \end{enumerate}
\end{proposition}

\begin{proof}
To prove statement~\eqref{enu:c3a}, take an arbitrary $g\in G\setminus A$. Noticing $G=\langle s\rangle A$, we assume that $g\in s^{k}A$ for some $k\in\{1,\ldots,q-2\}$. Thus $AgA=As^{k}A$ and $Ag^{-1}A=As^{-k}A$. If $AgA\ne Ag^{-1}A$, then $|A\{g,g^{-1}\}A|/|A|=2(|AgA|/|A|)$ and therefore $|A\{g,g^{-1}\}A|/|A|$ is even, satisfying statement~\eqref{enu:c3a}. Now assume $AgA=Ag^{-1}A$. Then
\[
As^{-k}A=As^{k}A=V\langle\varphi\rangle s^{k}V\langle\varphi\rangle=V\langle\varphi\rangle s^{k} \langle\varphi\rangle.
\]
Since $V\cap \langle s,\varphi\rangle=\{e\}$, it follows that $s^{-k}\in \langle\varphi\rangle s^{k} \langle\varphi\rangle$, and so $s^{-k}=\varphi^{i}s^{k}\varphi^{j}$ for some $i,j\in \{0,1,\ldots,f-1\}$. Then
$s^{-k}=\varphi^{i+j}(s^{k})^{\varphi^{j}}=\varphi^{i+j}s^{kp^{j}}$, which combined with $\langle s\rangle\cap \langle \varphi\rangle=1$ leads to $s^{k(p^{j}+1)}=e$. As a consequence,
\begin{equation}\label{eq:3}
    q-1\mid k(p^j+1).
\end{equation}
Note that $p^j+1$ divides $p^{jf}+1$ as $f$ is odd. We derive that each common divisor of $q-1$ and $p^j+1$ divides $2$, since $q-1=p^f-1$ divides $p^{jf}-1$. Since $p$ is odd, it follows that $\gcd(q-1,p^j+1)=2$. This together with~\eqref{eq:3} yields $(q-1)/2$ divides $k$. Then since $k\in\{1,\ldots,q-2\}$, we conclude that $k=(q-1)/2$. Consequently, $(s^k)^2=e\in H$. This proves statement~\eqref{enu:c3a}.

Noting that $|A|$ is odd, we deduce from~\cite[Theorem 3.6]{ZZ2021} that $H$ is a perfect code of $G$. Since $R(1)^s=R(\omega)\notin H$, we obtain that $H$ is a nonnormal subgroup of $G$. Moreover, it follows from $V\trianglelefteq G$ that $H^G\leq V^G=V\leq A$. Finally, seeing that $H^\varphi=\langle R(1)\rangle^\varphi=\langle R(1)\rangle=H$ and $H\trianglelefteq V$, we conclude that $A\leq \mathrm{N}_G(H)$. Therefore, it follows from Proposition~\ref{prop:closure} that $A$ is not a perfect code of $(G,H)$, as statement~\eqref{enu:c3b} asserts.
\end{proof}

In summary, the three infinite families of counterexamples constructed in this section illustrate the following corollary from three perspectives.

\begin{corollary}
    The converse of Theorem~$\ref{thm:1.2}$ does not hold. Consequently, the answer to Question~$\ref{ques:1}$ is negative.
\end{corollary}

\section{Perfect codes in $S_n$}\label{sec:4}

In this section, we first prove Proposition~\ref{prop:3}, giving the perfect code $S_m$ for the pair $(S_n,S_\ell)$ whenever $1\leq\ell<m<n$.
For a positive integer $k$, denote \[[k]=\{1,\ldots,k\}.\]

\begin{proof}[Proof of Proposition~$\ref{prop:3}$]
    Without loss of generality, assume that $S_n=\Sym([n])$, that $S_m$ is the pointwise stabilizer of $[n]\setminus[m]$ in $S_n$, and that $S_\ell$ is the pointwise stabilizer of $[n]\setminus[\ell]$ in $S_n$. Let
    \[
        \Omega=\{\sigma\colon[n]\setminus[m]\to[n]\mid \sigma\text{ is injective}\}
    \]
    be the set of injections from $[n]\setminus[m]$ to $[n]$. Take arbitrary $\sigma\in \Omega$ and $k\in [m]$, and regard $\sigma^0$ as the identity injection from $[n]\setminus [m]$ to $[n]$.

    We construct a sequence $k,k^{\sigma^{-1}},\ldots, k^{\sigma^{-j(\sigma,k)}}$ of pairwise distinct elements in $[n]$, where $j(\sigma,k)$ is a nonnegative integer, such that $k^{\sigma^{-(i-1)}}$ is the image of $k^{\sigma^{-i}}$ for each $i\in\{1,\ldots,j(\sigma,k)\}$ and
    \[
    k^{\sigma^{-j(\sigma,k)}}\notin ([n]\setminus[m])^\sigma.
    \]
    If $k\notin ([n]\setminus [m])^\sigma$, then let $j(\sigma,k)=0$. Otherwise, we have $k\in ([n]\setminus [m])^\sigma$, and so there exists a (unique) preimage $k^{\sigma^{-1}}\in[n]\setminus[m]$ of $k$ under $\sigma$. Moreover, since $k\in [m]$, we have $k\neq k^{\sigma^{-1}}$.
    Inductively, suppose that we have obtained pairwise distinct $k,k^{\sigma^{-1}},\ldots,k^{\sigma^{-s}}$ for some positive integer $s$ such that $k^{\sigma^{-(i-1)}}$ is the image of $k^{\sigma^{-i}}$ for each $i\in\{1,\ldots,s\}$. If $k^{\sigma^{-s}}\notin ([n]\setminus [m])^\sigma$, then let $j(\sigma,k)=s$. Otherwise, $k^{\sigma^{-s}}\in ([n]\setminus [m])^\sigma$, which indicates the existence of the preimage $k^{\sigma^{-(s+1)}}\in[n]\setminus[m]$ of $k^{\sigma^{-s}}$ under $\sigma$.
    Suppose for a contradiction that
    \begin{equation}\label{eq:4}
    k^{\sigma^{-(s+1)}}=k^{\sigma^{-t}}
    \end{equation}
    for some $t\in\{0,1,\ldots,s\}$. Since $k\in[m]$, it follows that $t\in\{1,\ldots,s\}$. Then applying $\sigma$ on both sides of~\eqref{eq:4} gives $k^{\sigma^{-s}}=k^{\sigma^{-(t-1)}}$, which contradicts the inductive hypothesis.
    Therefore, $k^{\sigma^{-1}},\ldots,k^{\sigma^{-(s+1)}}$ are pairwise distinct. Since $[n]$ is finite, the above inductive construction gives a nonnegative integer $j(\sigma,k)$ such that $k,k^{\sigma^{-1}},\ldots, k^{\sigma^{-j(\sigma,k)}}$ is the desired sequence.

    It is worth remarking in the paragraph above that $j(\sigma,k)=0$ if and only if $k\notin ([n]\setminus[m])^\sigma$, and that
    \begin{equation}\label{eq:range}
        k^{\sigma^{-s}}\in ([n]\setminus[m])^\sigma \ \text{ for each }\, s\in\{0,1,\ldots,j(\sigma,k)-1\}.
    \end{equation}
    We define a mapping $x(\sigma)\colon[n]\to[n]$ by letting
    \begin{equation}\label{eq:x_sigma}
        i^{x(\sigma)}=
        \begin{cases}
            i^\sigma&\text{ if }\ i\in [n]\setminus [m]\\
            i^{\sigma^{-j(\sigma,i)}}&\text{ if }\ i\in [m].
        \end{cases}
    \end{equation}

    We claim that $x(\sigma)\in S_n$. Assume that $a^{x(\sigma)}=b^{x(\sigma)}$ for some $a,b\in [n]$. Since $([n]\setminus [m])^{x(\sigma)}=([n]\setminus [m])^\sigma$ and $[m]^{x(\sigma)}\cap ([n]\setminus [m])^\sigma=\emptyset$, either $a,b\in [n]\setminus [m]$ or $a,b\in [m]$. For the former, we obtain $a=b$ by the injectivity of $\sigma$. Now assume that $a,b\in [m]$. Then $a^{\sigma^{-j(\sigma,a)}}=a^{x(\sigma)}=b^{x(\sigma)}=b^{\sigma^{-j(\sigma,b)}}$. Without loss of generality, assume that $j(\sigma,b)\leq j(\sigma,a)$. It follows that $a^{\sigma^{-j(\sigma,a)+j(\sigma,b)}}=b$. Note by~\eqref{eq:range} that $a^{\sigma^{-i}}\in [n]\setminus [m]$ for each $i\in\{1,\ldots,j(\sigma,a)\}$. This together with $b\in[m]$ forces $j(\sigma,a)=j(\sigma,b)$ and hence $a=b$. Thus, $x(\sigma)$ is a permutation on $[n]$, as claimed.

    Let $X=\{x(\sigma) \mid \sigma\in\Omega\}$. Assume that $x(\sigma_1)S_m=x(\sigma_2)S_m$. Then $x(\sigma_1)x(\sigma_2)^{-1}\in S_m$ fixes each $k\in [n]\setminus [m]$. Consequently, $k^{\sigma_1}=k^{x(\sigma_1)}=k^{x(\sigma_2)}=k^{\sigma_2}$ for all $k\in [n]\setminus[m]$, which means $\sigma_1=\sigma_2$. This shows that distinct $\sigma$ in $\Omega$ gives distinct $x(\sigma)$ and further distinct $x(\sigma)S_m$. Then since $|X|=|\Omega|=n!/m!=|S_n|/|S_m|$, we see that $X$ is a left transversal of $S_m$ in $S_n$.
    Hence, by Theorem~\ref{thm:pc}, it remains to prove $XS_\ell=S_\ell X^{-1}$. Noting that $|XS_\ell|=|(XS_\ell)^{-1}|=|S_\ell X^{-1}|$, we only need to show that $XS_\ell\subseteq S_\ell X^{-1}$.

    We first present the cycle decomposition of $x(\sigma)$ with $\sigma\in\Omega$. Let
    \[
        I(\sigma)=\{k\in [m]\mid j(\sigma,k)>0\}.
    \]
    For each $k\in [m]\setminus I(\sigma)$, we have $k^{x(\sigma)}=k^{\sigma^0}=k$. Note by~\eqref{eq:range} that $k^{\sigma^{-j(\sigma,k)}},\ldots,k^{\sigma^{-1}}\in [n]\setminus [m]$. Thus, for each $k\in I(\sigma)$, the orbit of $\langle x(\sigma) \rangle$ containing $k$ is $\{k^{\sigma^{-j(\sigma,k)}},\ldots,k^{\sigma^{-1}},k\}$. Therefore, the cycle decomposition of $x(\sigma)$ can be written as
    \begin{equation}\label{eq:cycle}
        x(\sigma)=c_1\cdots c_s\prod_{k\in I(\sigma)}(k^{\sigma^{-j(\sigma,k)}},\ldots,k^{\sigma^{-1}},k),
    \end{equation}
    where $c_1,\ldots,c_s$ are disjoint cycles in $\Sym([n]\setminus[m])$ for some integer $s\geq 0$. In particular, each cycle in~\eqref{eq:cycle} contains at most one point in $[m]$, and $x(\sigma)$ fixes $[m]\setminus I(\sigma)$ pointwise.

    Take arbitrary $\sigma\in\Omega$ and $h\in S_\ell$. Assume that the cycle decomposition of $x(\sigma)$ is as in~\eqref{eq:cycle}. Then define $y(\sigma)$ to be the element of $S_n$ with cycle decomposition
    \begin{equation}\label{eq:y}
        y(\sigma)=c_1^{-1}\cdots c_s^{-1}\prod_{k\in I(\sigma)}(k^{\sigma^{-1}},\ldots,k^{\sigma^{-j(\sigma,k)}},k^h).
    \end{equation}
    Note that, since $h\in S_\ell<S_m$ stabilizes $[m]$, the right-hand side of~\eqref{eq:y} is indeed a cycle decomposition.

    We claim that $y(\sigma)\in X$. Let $\tau$ be the restriction of $y(\sigma)$ on $[n]\setminus[m]$. Then $\tau\in \Omega$. We verify the claim by showing $y(\sigma)=x(\tau)$. For $i\in [n]\setminus[m]$, we have $i^{y(\sigma)}=i^{\tau}=i^{x(\tau)}$ by~\eqref{eq:x_sigma}. Take an arbitrary $i\in [m]$. Since $h\in S_\ell<S_m$, we have $i=k^h$ for some $k\in[m]$.
    Since $\tau$ is the restriction of $y(\sigma)$ on $[n]\setminus[m]$, it follows from~\eqref{eq:y} that for the sequence $i,i^{\tau^{-1}},\ldots,i^{\tau^{-j(\tau,i)}}$ is in fact $k^h,k^{\sigma^{-j(\sigma,k)}},\ldots,k^{\sigma^{-1}}$. In particular, either $j(\sigma,k)>0$ and $i^{\tau^{-j(\tau,i)}}=k^{\sigma^{-1}}$, or $j(\sigma,k)=j(\tau,i)=0$. For the former,~\eqref{eq:y} yields $i^{y(\sigma)}=k^{\sigma^{-1}}=i^{\tau^{-j(\tau,i)}}$. For the lattere, $k\notin I(\sigma)$, and~\eqref{eq:y} shows $i^{y(\sigma)}=i=i^{\tau^0}=i^{\tau^{-j(\tau,i)}}$. In either case, we obtain $i^{y(\sigma)}=i^{\tau^{-j(\tau,i)}}=i^{x(\tau)}$ by~\eqref{eq:x_sigma}. Therefore, $y(\sigma)=x(\tau)$, which proves the claim.

    Now we complete the proof of $XS_\ell\subseteq S_\ell X^{-1}$ by demonstrating that $x(\sigma)hy(\sigma)\in S_\ell$, or equivalently, $x(\sigma)hy(\sigma)$ fixes every point in $[n]\setminus[\ell]$. Take an arbitrary $i\in [n]\setminus[\ell]$. Then we have $i^h=i$. Assume first that $i^{x(\sigma)}\in [\ell]\subseteq [m]$. In this case, it follows from~\eqref{eq:cycle} that either $i^{x(\sigma)}=i\in [m]\setminus I(\sigma)$ or $i^{x(\sigma)}\in I(\sigma)$. For the former, $i^{x(\sigma)hy(\sigma)}=i^{hy(\sigma)}=i^{y(\sigma)}=i$ by~\eqref{eq:y}. For the latter, we derive from~\eqref{eq:cycle} that $i\in [n]\setminus [m]$ and then derive from~\eqref{eq:y} and~\eqref{eq:x_sigma} that
    \[
    i^{x(\sigma)hy(\sigma)}=(i^{x(\sigma)})^{hy(\sigma)}=(i^{x(\sigma)})^{\sigma^{-1}}=(i^\sigma)^{\sigma^{-1}}=i.
    \]
    Assume next that $i^{x(\sigma)}\notin [\ell]$. In this case, $(i^{x(\sigma)})^h=i^{x(\sigma)}$, and so $i^{x(\sigma)hy(\sigma)}=(i^{x(\sigma)})^{y(\sigma)}$. By~\eqref{eq:cycle} and~\eqref{eq:y}, either $(i^{x(\sigma)})^{y(\sigma)}=i$, or $i=k^{\sigma^{-1}}$ and $i^{x(\sigma)}=k$ for some $k\in I(\sigma)$. Note for the latter that $k^h=k$, which in conjunction with~\eqref{eq:y} implies 
    \[
    (i^{x(\sigma)})^{y(\sigma)}=k^{y(\sigma)}=(k^h)^{y(\sigma)}=k^{\sigma^{-1}}=i. 
    \]
    We conclude that $i^{x(\sigma)hy(\sigma)}=(i^{x(\sigma)})^{y(\sigma)}=i$. Thus, $x(\sigma)hy(\sigma)$ fixes every $i\in[n]\setminus[\ell]$, as desired.
\end{proof}

In the remainder of this section, let us initiate the determination of which maximal subgroups of $S_n$ are perfect codes of $S_n$. Recall that each intransitive maximal subgroup of $S_n$ has the form $S_m\times S_{n-m}$ for some $m\in\{1,\ldots,n-1\}$ such that $n\neq2m$. The following proposition shows that all intransitive maximal groups of $S_n$ are perfect codes of $S_n$.

\begin{proposition}\label{prop:intransitive}
	Let $G=\Sym([n])=S_n$ and $A=\Sym([m])\times\Sym([n]\setminus[m])=S_m\times S_{n-m}$ with $1\leq m\leq n-1$. Then $A$ is a perfect code of $G$.
\end{proposition}

\begin{proof}
Take an arbitrary $x\in G$ such that $x^2\in A$. Let $y$ be the product of elements in the set
\[
I=\{(k,k^x)\mid k\in [m]\;\text{and}\; k^x\notin[m]\}.
\]
Here, if $I=\emptyset$, we define $y=e$. Since the $2$-cycles $(k,k^x)$ are pairwise disjoint, it follows that $y^2=e$. Moreover, since $yx^{-1}$ maps elements of $[m]$ into $[m]$, we have $y\in Ax$. Therefore, Theorem~\ref{thm:1} implies that $A$ is a perfect code of $G$.
\end{proof}

Next, we study which maximal affine groups of $S_n$ are perfect codes of $S_n$. To obtain our partial result, we need~\cite[Theorem~1.1]{Z2023}, which shows that subgroup perfect codes are essentially determined by their Sylow $2$-subgroups. In fact, this result states that a subgroup $H$ of a finite group $G$ is a perfect code of $G$ if and only if $H$ has a Sylow $2$-subgroup that is a perfect code of $G$. Another ingredient needed is the following lemma.

\begin{lemma}\label{lm:normaliser}
    Let $G$ be a permutation group of degree $2^{t}s$, where $t$ is a positive integer and $s$ is an odd integer, such that $G$ has a semiregular subgroup $Q$ of order $2^{t}$. Then $xQ$ contains an involution for every $x\in \mathrm{N}_{G}(Q)\setminus Q$ with $x^2\in Q$.
\end{lemma}

\begin{proof}
    Let $x\in \mathrm{N}_{G}(Q)\setminus Q$ such that $x^2\in Q$, and let $H=\langle Q,x\rangle$. Then $H$ is of order $2^{t+1}$. Since $Q$ is a semiregular group of order $2^{t}$, every orbit of $H$ is of length $2^{t}$ or $2^{t+1}$. Moreover, since the degree of $G$ is $2^ts$ with $s$ odd, $H$ has at least one orbit of length $2^{t}$. Take a point $\alpha$ from such an orbit. Then the stabilizer of $\alpha$ in $H$ has order $2$, which means that there exists an involution $y\in H$ satisfying $\alpha^y=\alpha$. Since $Q$ is semiregular, we have $y\notin Q$, and so $y\in xQ$. This completes the proof.
\end{proof}

We are now ready to prove our conclusion regarding affine perfect codes.

\begin{proposition}\label{prop:affine}
    Let $G=\Sym(\bbF_p)=S_p$ and $H=\AGL(1,p)\leq G$, where $p$ is an odd prime. Then $H$ is a perfect code of $G$.
\end{proposition}

\begin{proof}
    Let $Q$ be a Sylow 2-subgroup of the stabilizer $H_0=\GL(1,p)$. As $H_0$ is semiregular on $\bbF_p\setminus\{0\}$, so is $Q$. Since $|H|/|H_0|=p$ is odd, it follows that $Q$ is also a Sylow 2-subgroup of $H$. By~\cite[Theorem~1.1]{Z2023}, it suffices to show that $Q$ is a perfect code of $G$.
    Take an arbitrary $x\in G$ such that $x^2\in Q$ and $|Q|/|Q\cap Q^x|$ is odd. By Theorem~\ref{thm:1}, we are left to prove that there exists $y\in xQ$ with $y^2=e$. As the conclusion is obvious when $x\in Q$, we may assume $x\notin Q$.

    Since $Q$ is a $2$-group while $|Q|/|Q\cap Q^x|$ is odd, we have $|Q|/|Q\cap Q^x|=1$, which means $x\in\mathrm{N}_G(Q)$.
    First assume that $x$ lies in $G_0$, the stabilizer of $0$ in $G$. Then $x\in \mathrm{N}_{G_0}(Q)\setminus Q$. Regarding $G_0$ and $Q$ as permutation groups on $\bbF_p\setminus\{0\}$, we derive from Lemma~\ref{lm:normaliser} that there exists $y\in xQ$ such that $y^2=e$, as desired.
    Next assume that $x\notin G_0$. In this case, $0^x=k$ for some $k\in\bbF_p\setminus\{0\}$. It then follows from $x^2\in Q\leq G_0$ that $0=0^{x^2}=k^x$, and thus, $k=0^x=k^{x^2}$. Since $Q$ is semiregular on $\bbF_p\setminus\{0\}$, we conclude that $x^2=e$. This completes the proof.
\end{proof}



\section*{Acknowledgements}

The first author acknowledges the support of ARC Discovery Project DP250104965. The second author was supported by the Natural Science Foundation of Chongqing (CSTB2022NSCQ-MSX1054). The third author was supported by the Melbourne Research Scholarship provided by The University of Melbourne.

\end{document}